\documentclass[11pt]{article}

\usepackage{amsmath,amscd,amsfonts,amsthm}

\newcommand{\shrinkmargins}[1]{
  \addtolength{\textheight}{#1\topmargin}
  \addtolength{\textheight}{#1\topmargin}
  \addtolength{\textwidth}{#1\oddsidemargin}
  \addtolength{\textwidth}{#1\evensidemargin}
  \addtolength{\topmargin}{-#1\topmargin}
  \addtolength{\oddsidemargin}{-#1\oddsidemargin}
  \addtolength{\evensidemargin}{-#1\evensidemargin}
  }

\shrinkmargins{.7}

\newcommand{\field}[1]{\mathbb{#1}}

\newcommand{\F}{\field{F}}

\newcommand{\R}{\field{R}}
\newcommand{\A}{\field{A}}

\renewcommand{\P}{\field{P}}

\newcommand{\beq}{\begin{displaymath}}
\newcommand{\eeq}{\end{displaymath}}
\newcommand{\beqn}{\begin{equation}}
\newcommand{\eeqn}{\end{equation}}

\theoremstyle{plain}
\newtheorem{thm}{Theorem}

\newtheorem{cor}[thm]{Corollary}
\newtheorem{lem}[thm]{Lemma}

\theoremstyle{definition}

\theoremstyle{remark}
\newtheorem{rem}[thm]{Remark}
\newtheorem{question}[thm]{question}

\title{An incidence conjecture of Bourgain over fields of positive characteristic}
\author{Jordan S. Ellenberg and M\'{a}rton Hablicsek}
\begin{document}

\maketitle

The goal of this note is twofold:  first, to generalize a recent theorem of Guth and Katz on incidences between points and lines in $3$-space from characteristic $0$ to characteristic $p$, and second, to explain how some of the special features of algebraic geometry in characteristic $p$ manifest themselves in problems of incidence geometry.

Let $X$ be a reduced and irreducible hypersurface in $\A^n$, i.e. the variety cut out by the vanishing of some irreuducible polynomial $F(x_1, \ldots, x_n)$.  We say that $F$ is {\em flexy} if, for every smooth point of $x$, the tangent plane to $X$ at $x$ meets $X$ with a degree of tangency greater than two.  For instance, to say a plane curve is flexy is to say that every smooth point on the curve is an inflection point.  Of course, for a plane curve over the real numbers, this implies that the curve is a line.  But this is not the case in characteristic $p$.  For instance, the curve with equation $x^3 y + y^3 z + z^3 x$ over $\F_3$ is flexy; this is the curve famously called ``the funny curve" by Hartshorne.

Since we are working over a field of characteristic $p$, the notions of ``tangency" and ``differential" relevant here are the algebraic versions.  In particular, to say $X$ passes through the origin is to say $F$ vanishes at the origin, and thus has a Taylor series there that starts:
\beq
F(x_1, \ldots, x_n) = 0 + F_1(x_1, \ldots, x_n) + F_2(x_1, \ldots, x_n) + h.o.t.
\eeq
where $F_i$ is a homogenous polynomial of degree $i$.  To say that $X$ is flexy at the origin is precisely to say that $F_2$ is divisible by $F_1$.  (The ``flexy" points here are precisely those points which are called ``flat points" by Guth and Katz; we have avoided the word ``flat" here in order to avoid conflict with its other uses in arithmetic geometry.)
  
\begin{thm}  Let $k$ be a field and let $L$ be a set of $N^2$ lines in $k^3$, such that
\begin{itemize}
\item  no more than $2N$ lines lie in any plane;
\item  no more than $2Nd$ lines lie in any flexy surface of degree $d$.
\end{itemize}
Let $S$ be a set of points such that each line in $L$ contains at least $N$ points of $S$.  Then $|S| > cN^3$ for some absolute constant $c$.
\label{th:main}
\end{thm}

\begin{rem} The notion of ``flexiness" is very closely related to that of {\em non-reflexivity}, the failure of the map from $X$ to its dual variety to be generically smooth; indeed, for curves in odd characteristic the two notions are the same, as \cite[Thm 5.90]{hirschfeld} shows.  We have chosen to use the less standard criterion ``flexiness" on the grounds that it is simpler to describe and fits more naturally into the proof of the theorem.
\end{rem}

In characteristic $0$, every flexy variety is a plane.   In that case, Theorem~\ref{th:main} is an assertion about a family of $N^2$ lines, no $N$ of which are contained in any plane.  This latter condition can be thought of as a form of the {\em Wolff axiom}.  So, when $k$ has characteristic $0$, the conclusion of Theorem~\ref{th:main} is a theorem of Guth and Katz~\cite[Theorem 2]{guth:guthkatz}, which settles a conjecture of Bourgain.  (The theorem in \cite{guth:guthkatz} is stated for $k = \R$, but the proof works word for word over any field of characteristic $0$.)

In characteristic $p$, however, Theorem~\ref{th:main} does not hold without the restriction concerning lines lying in a flexy surface.  For example, take $k = \F_{p^2}$ and let $X \subset \A^3$ be the ``Heisenberg surface" cut out by the equation
\beq
x - x^p + yz^p - zy^p = 0.
\eeq
Then the lines of the form $\{(a,b,0)+t(\bar{b},v,1)|t\in \F_{p^2}\}$, where $a$ and $v$ both lie in $\F_p$ and $\bar{b}$ denotes the Galois conjugate $b^p$ of $b$, lie in $X$.  There are $p^4$ such lines; let $L$ be this set of lines.  Now take $S$ to be the set $X(\F_{p^2})$ and take $N = p^2$.  The intersection of any plane with $X$ is a curve of degree $p$, which can contain at most $p$ (that is, $N^{1/2}$) lines.  But $S$ clearly contains all $N$ of the $\F_{p^2}$-rational points on each of the lines in $L$.  Finally, one can check that $|S| \sim N^{5/2}$; so this set does not conform to the conclusion of Theorem~\ref{th:main}.  However, $X$ is a flexy surface of degree $N^{1/2}+1$, in which all $N^2$ of the lines are contained; so this counterexample is excluded by the hypothesis of Theorem~\ref{th:main}.

We note that the Heisenberg surface is precisely the one that appears in the paper of Mockenhaupt and Tao~\cite[\S 8]{mocktao} as an example of a set $P$ of points in $3$-dimensional space over a finite field $\F$ which contains $|\F|^2$ lines, no $|\F|$ of which are contained in a plane, but which has cardinality much less than $|\F|^3$.  The Mockenhaupt-Tao paper concerned Kakeya sets in $\F^3$:  subsets containing a line in each of the $\sim |\F|^2$ possible directions.  Kakeya sets in $\F^3$ are now known, by Dvir's theorem~\cite{dvir}, to have cardinality on order $|\F|^3$.  The Heisenberg surface shows that, by contrast, there are much smaller subsets of $\F^3$ which satisfy the Wolff axiom and which contain $|\F|^2$ lines, once we relax the condition that these lines all point in different directions.

In the important case where $k = \F$ is a finite field and $N = |\F|$, Theorem~\ref{th:main} can be seen as a strengthening of Dvir's theorem in the $3$-dimensional case.  Suppose $L$ is a set of $N^2$ lines in $\F^3$ which satisfies the Kakeya condition; the lines all point in distinct directions.  We may think of $\F^3$ as being affine space embedded in projective space $\P^3(\F)$.  Take $H$ to be the plane at infinity; then the Kakeya condition can be rephrased as saying that the lines in $L$ intersect the plane $H$ at infinity in $N^2$ distinct points.  If $X$ is a hypersurface of degree $d$, then any line in $L$ which is contained in $X$ must intersect $H$ somewhere on the curve $X_0 = X \cap H$.  Since $X_0$ is a degree-$d$ plane curve, it has at most $d(\F|+1)$ points; thus, at most $d(|\F|+1)$ of the lines can be contained in the hypersurface $X$.  In particular, the Kakeya condition implies the conditions of Theorem~\ref{th:main}.  But the weaker onditions of Theorem~\ref{th:main} already suffice to guarantee that the union of the $|\F|^2$ lines contains a positive proportion of the points of $\F^3$.  Our point of view is that the conditions of Theorem 1 should be thought of as the appropriate modification of the Wolff axiom to use in a characteristic $p$ context.

When $\F$ is a prime field $\F_p$, the situation is even more agreeable.  We note that flexy surfaces which are not planes have to have degree at least as great as the characteristic:

\begin{lem}  Let $k$ be a field and let $X \in \A^n/k$ be a (reduced, irreducible) flexy hypersurface of degree $d > 1$.  Then $d \geq p$.
\label{le:degree}
\end{lem}

\begin{proof}  
Without loss of generality we assume $k$ is algebraically closed.  It is immediate that any hyperplane section of a flexy variety is a flexy subvariety.  Choose a plane in $\A^n$ whose intersection with $X$ is an irreducible curve; then $C$ is a flexy plane curve of degree $d > 1$, which is well-known to have degree at least $p$ (see e.g. \cite[Thm 5.90]{hirschfeld}.)
\end{proof}

We remark that the real arithmetic content of Lemma~\ref{le:degree} is that, when $X$ is flexy, the map from $X$ to its dual is not generically smooth, which means that it must be inseparable, which means that its degree must be a multiple of $p$.

%Without loss of generality we assume $k$ is algebraically closed.  If $\deg X = 1$ the reflexivity is clear, so we assume $X$ is not a plane from now on.

%Since $\phi: X^{sm} \ra X^*$ is not generically smooth, it is not separable.  In particular, the preimage of a point $Q \in X^*(k)$ is a non-reduced subscheme of $X$ in which every component has multiplicity divisible by $p$.  Let $H_Q$ be the hyperplane in $\P^n$ corresponding to $Q$ and let $P \in X(k)$ be a point in $\phi^{-1}(Q)$.  Then $H_Q \cap X$ is a curve $C$ in $H_Q$ with multiplicity at least $p$ at $P$.  In particular, $\deg C$ is at least $p$, whence $\deg X$ is at least $p$.

\medskip

From Lemma~\ref{le:degree} and Theorem~\ref{th:main} one immediately obtains the following corollary:

\begin{cor} Let $L$ be a set of $p^2$ lines in $\F_p^3$, no more than $p$ of which lie in any plane.  Then the union of all the lines in $L$ has cardinality at least $cp^3$ for some absolute constant $c$.
\label{co:fp}
\end{cor}

\begin{proof}
Apply Theorem~\ref{th:main} with $N = p$, noting that the second part of the hypothesis is vacuous, since $2Nd > 2p^2 > |L|$ for any flexy surface which is not a plane.
\end{proof}

In other words, the conclusion of Guth and Katz regarding incidences of lines and points over $\R$ remains true as an assertion about lines and points over $\F_p$, while it is false, as witnessed by the Heisenberg surface, when $k = \F_{p^2}$.  

One might ask over which finite fields the analogue of Corollary~\ref{co:fp} holds; in fact, it is true {\em only} for finite fields of prime order, as we now demonstrate.

Specifically, we will show that the hypersurface $X$ of $\F_{p^n}^3$ cut out by the polynomial
$$f(x,y,z)=x+x^p+\dots +x^{p^{n-1}}+yz^p+y^pz^{p^2}+\dots+y^{p^{n-1}}z-yz^{p^{n-1}}-y^pz-\dots- y^{p^{n-1}}z^{p^{n-2}}$$
contains $p^{2n}$ lines, but $|X(\F_{p^n})|=p^{3n-1}$.

First of all, we note that the expression $x + \ldots + x^{p^{n-1}}$, as $x$ ranges over $\F_{q^n}$, takes each value in $\F_q$ exactly $q^{n-1}$ times.  It follows that $f$, considered as a map from $\F_{q^n}^3$ to $\F_q$, takes each value $q^{3n-1}$ times.  In particular, $|X(\F_{p^n})|=p^{3n-1}$. Next, we show that the surface contains at least $p^{2n}$ lines. We are only interested in lines which intersect the $xy$-plane transversely, i.e., which are of the form $L_{(a,b,u,v)}=\{(a,b,0)+t(u,v,1)|t\in \F_{p^n}\}$. Note that the values of $a,b,u,v$ uniquely determine the line.

If $L_{(a,b,u,v)}\subset X$, then the triples $x=a+tu$, $y=b+tv$, $z=t$ are solutions for $f(x,y,z)$ for any value of $t$. Therefore the coefficients of the $t^j$'s vanish. It is straightforward to check that the coefficients $c_l$ of $t^l$ can be characterized as follows:
\begin{itemize}
\item if $l=p^j+p^{j-1}$, then $c_l=c_{p^j+p^{j-1}}=v^{p^{j-1}}-v^{p^{j}}$,
\item if $l=p^j$, then $c_l=c_{p^j}=b^{p^{j-1}}-b^{p^{j+1}}+u^{p^j}$,
\item if $l=0$, then $c_l=c_0=a+a^p+\dots +a^{p^{n-1}}$, and
\item $c_l=0$, otherwise.
\end{itemize}
Note that $c_{p^i}^p=c_{p^{i+1}}$, therefore $c_{p^i}$ vanishes if and only if $c_{p^j}$ vanishes. Similarly, $c_{p^i+p^{i-1}}=0$ if and only if $c_{p^{i+1}+p^i}=0$. As a consequence $L_{(a,b,u,v)}\subset X$ if and only if
\begin{itemize}
\item $v-v^p=0$,
\item $b-b^{p^2}+u^p=0$,
\item $a+a^p+\dots a^{p^{n-1}}=0$.
\end{itemize}
So a line is given by one of the $p$ choices for $v$, one of the $p^n$ choices for $b$ (which determines $u$) and one of the $p^{n-1}$ choices for $a$.  So the number of lines of the form $L_{(a,b,u,v)}$ contained in $X$ is $p\cdot p^n\cdot p^{n-1}=p^{2n}$.

\begin{question}:  Arguing as above, one can show that, if $q$ is a prime power $p^m$ with $m > 1$, there is a set of $q^2$ lines in $\F_q^3$, no $q$ contained in a plane, whose union has cardinality $\sim q^{3-1/d}$, where $d$ is the smallest nontrivial divisor of $m$.   Is this sharp?  (The argument of Mockenhaupt and Tao shows that the union can be no smaller than $q^{5/2}$, so the bound is sharp when $m$ is even.)  This question might be approachable by a more refined description of flexy surfaces of low degree.
\end{question}

\bigskip

We now prove Theorem~\ref{th:main}.

\bigskip

\begin{proof}

The following lemma is unchanged from Guth-Katz (see the proof of Theorem 1.2 of \cite{guth:guthkatz}), which we add for the sake of completeness.

\begin{lem} Let $L$ be a set of $N^2$ lines in $k^3$ and let $S$ be a set of points such that each line in $L$ contains at least $N$ points of $S$. Suppose $|S|=\frac{N^3}{K}$, where $K$ is a sufficiently large constant. Then there exists
\begin{itemize}
\item an irreducible hypersurface $X$ of degree $d\leq \frac{N}{4}$,
\item a subset $S'\subset S\cap X(k)$,
\item subsets $L''\subset L'\subset L$, with $|L''|\geq 2Nd$.
\end{itemize}
such that
\begin{itemize}
\item each point on $S'$ is on at least 3 lines of $L'$,
\item each line in $L''$ contains at least 10$d$ points of $S'$.
\end{itemize}
\label{est}
\end{lem}
\medskip

\begin{proof}

We assume, again without loss of generality, that $k$ is algebraically closed.

If $I$ is the set of incidences, i.e. the set of pairs $(p,\ell)$ with $p$ a point in $S$ and $\ell$ a line in $L$ containing $p$, then $I$ contains at least $N$ points projecting to any given line.  We distinguish a subset $I'$ of $I$ containing {\em exactly} $N$ incidences for each line, and from now on use the word ``incidence" to refer only to these distinguished incidences.  In particular, if $T$ is a subset of $S$ and $M$ a subset of $L$, we denote by $I(T,M)$ the number of incidences $(p,\ell)$ in $I'$ with $p \in T$ and $\ell \in M$.  So $I(S,L) = |L|\cdot N = N^3$.

%
%Since each line is incident to at least $N$ points of $S$, we can color exactly $N$ points of $S$ on each line, so that whenever we calculate incidences between some subset of $S$ and some subset of $L$, we only consider the colored points. We denote by $I(T,K)$ the number of incidences between a subset $T$ of $S$ and a subset $K$ of $L$. As a consequence of the above discussion $I(S,L)=|L|\cdot N=N^3$.

We define $v(x)$ to be the number of lines incident to $x$. We denote by $S_v$ the set of points $x$ of $S$ such that $v(x)\geq \frac{K}{1000}$. Each line is incident to exactly $N$ points, therefore 
\[I(S\setminus S_v,L)\leq |S|\cdot \frac{K}{1000}=\frac{N^3}{K}\cdot \frac{K}{1000}=\frac{N^3}{1000}\]
implying that $I(S_v,L)\geq 999N^3/1000$. 

We define similarly the sets $S_j$ to be the sets of points $x\in S$ such that $\frac{2^{j-1}K}{1000}\leq v(x)\leq \frac{2^jK}{1000}$. Since, $\sum_{j=1}^{\infty} I(S_j,L)\geq I(S_v,L)\geq \frac{999N^3}{1000}$ and $\sum_{j=1}^{\infty} \frac{1}{j^2}<2$, by the pigeonhole principle, there exists a $j \geq 1$ such that $I(S_j,L)\geq \frac{999N^3}{2000j^2}$. Since for each element $x$ of $S_j$, $v(x)\leq \frac{2^jK}{1000}$, we obtain 
\[|S_j|\cdot  \frac{2^{j}K}{1000}\geq I(S_j,L)\geq \frac{999N^3}{2000j^2},\]
implying that
\[|S_j|\geq\frac{999N^3}{2K2^{j}j^2}.\]

Similarly, since $\frac{2^{j-1}K}{1000}\leq v(x)$ for each element of $S_j$, thus $N^3=I(S,L)\geq I(S_j,L)\geq |S_j|\cdot \frac{2^{j-1}K}{1000}$, implying
\[\frac{2000N^3}{K2^j}\geq |S_j|\geq \frac{999N^3}{2K2^{j}j^2}.\]
For any set $T\subset k^3$ of size at most ${d+3 \choose 3}$, there exists a polynomial of degree at most $d$ vanishing on the points of $T$. Since $|S_j|\leq \frac{2000N^3}{K2^j}$, there exists a polynomial $P$ of degree at most $\frac{25N}{K^{1/3}2^{j/3}}$ vanishing on $S_j$. We can assume that $P$ is square-free and separable.  It may not be irreducible; if it not, we may factor it into irreducible factors, $P=P_1P_2\dots P_m$. We denote the degrees of the $P_l$ by $d_l$. Let $S_{j,l}$ be the set of points of $S_j$ where $P_l$ vanishes. We have
\[\sum_{l=1}^m d_l\leq \frac{25N}{K^{1/3}2^{j/3}}
\mbox{ and } 
\sum_{l=1}^m |S_{j,l}|\geq \frac{999N^3}{2K2^jj^2}.\]
Again, by the pigeonhole principle, we can find an $l$ such that
$$|S_{j,l}|\geq \frac{999N^2d_l}{50K^{2/3}2^{2j/3}j^2}.$$
We denote by $X$ the hypersurface cut out by $P_l$, and by $d$ the corresponding degree $d_l$.  Note that $S_{j,l} \subset X(k)$.

We denote by $L'$ the set of lines in $L$ incident to more than $100d$ points of $X(k)$. Clearly, $L'\subset L\cap X(k)$ and
\[I(S_{j,l},L\setminus L')\leq |L'|100d\leq |L|100d=100N^2d.\]
A similar calculation shows that if $S'$ denotes the set of points of $S_{j,l}$ incident to at least 3 lines of $L'$, then we have $I(S_{j,l}\setminus S',L')\leq 2|S_{j,l}|.$ Finally, if $L''$ denotes the set of lines in $L'$ incident to more than $10d$ points of $S'$, then $I(S',L'\setminus L'')\leq |L|10d=10N^2d$.
Combining the above inequalities we have
\begin{align*}I(S',L'')&\geq I(S_{j,l},L)-I(S_{j,l},L\setminus L')-I(S_{j,l}\setminus S',L')-I(S',L'\setminus L'')\geq\\
&\geq |S_{j,l}|\frac{2^{j-1}K}{1000}-100N^2d-2|S_{j,l}|-10N^2d.\end{align*}
By definition $|S_{j,l}|\geq \frac{999N^2d}{50K^{2/3}2^{2j/3}j^2}$, thus we can choose a sufficiently large $K$ so that $I(S',L'')\geq 2N^2d$. Since every line in $L''$ is incident to at most $N$ points of $S'$, we obtain $|L''|\geq 2Nd$.

By taking a possibly larger $K$, we can ensure that $d\leq \frac{N}{4}$.
\end{proof}

%Following Guth and Katz, we call a point $P \in X(k)$ a {\em flat point} if the intersection of $P$ with the tangent plane to $X$ at $P$ has multiplicity at least $3$.  

\begin{lem}  Let $X$ be a reduced irreducible non-flexy surface of degree $d>1$ in $\A^3$.  Let $L_1$ be the set of lines contained in $X$ which contain at least $d$ singular points of $X$, and let $L_2$ be the set of lines contained in $X$ which contain at least $3d-3$ flexy points of $X$.  Then
\beq
|L_1| + |L_2| < 4d^2.
\eeq
\label{le:badlines}
\end{lem}

\begin{proof}
Let $F$ be an irreducible squarefree polynomial such that $V(F) = X$.

The singular locus of $X$ is cut out by the vanishing of $F$ and its first partial derivatives, which are of degree $d-1$.  So $F$ and all the partial derivatives vanish identically on every line in $L_1$.  Since $X$ is reduced, the singular locus is a curve $C_1$ in $X$, and since $C_1$ is contained in the intersection of $X$ with one of its partial derivatives, we have $\deg C_1 \leq d(d-1)$.  In particular, no more than $d(d-1)$ lines can be contained in $C_1$, so $|L_1| \leq d(d-1)$.

Let $C_2$ be the locus of flexy points of $X$.  By hypothesis, $C_2$ is a proper subvariety of $X$.  Let $p$ be a non-flexy point of $X$; then the generic hyperplane section of $X$ containing $p$ is a non-flexy curve.  Choose a hyperplane $H$ such that $H \cap X$ is a non-flexy curve $Z$ in $H$.  By change of coordinates, we may assume that $p$ is the origin and $H$ is the plane $z=0$.  So $Z$ is simply the plane curve cut out by the vanishing of $F(x,y,0)$.  At any flexy point of $X$, one has
\beq
G := F_x^2 \frac{F_{yy}}{2} + F_y^2 \frac{F_{xx}}{2} - F_x F_y F_{xy} = 0
\eeq
where $F_x$ (resp. $F_{xy}$) denotes the partial derivative  (resp. second partial derivative) of $F$ with respect to  $x$ (resp. with respect to $x,y$.) In characteristic 2, we define $\frac{F_{xx}}{2}$ to be the divided power operation, in other words $\frac{x^n_{xx}}{2}$ is defined to be $\frac{n(n+1)}{2}x^{n-2}$.

The locus $C_2$ is contained in the locus where both $F$ and $G$ vanish, and the non-flexiness of $Z$ implies that $G$ is not a multiple of $F$; thus $F \cap G$ is a curve of degree $(\deg F)(\deg G) = d(3d-4)$, which contains $C_2$.  The restriction of $G$ to any line in $L_2$ is a polynomial of degree $3d-4$ which vanishes at at least $3d-3$ distinct points, and is thus $0$: so all the lines in $L_2$ are contained in the curve $F \cap G$, which implies that  $|L_2| \leq d(3d-4)$.

\end{proof}

Assume that $|S|<\frac{N^3}{K}$ for a large enough $K$. Take an irreducible hypersurface $X$ of degree $d\leq\frac{N}{4}$ given by Lemma~\ref{est}.  Each line in $L''$ contains at least $10d$ points of $X$, and is thus contained in $X$ as a variety.  And each point in $S'$ is contained in at least three such lines, which implies that it is either a singular point or a flexy point on $X$.  Therefore each point of $S'$ is either a singular point or a flexy point on $X$. It now follows by Lemma~\ref{est} that $|L''|\geq 2Nd\geq 4d^2$. By Lemma~\ref{le:badlines}, $X$ is either a plane or a flexy surface.  But $X$ contains at least $2Nd$ lines of $L$, which violates the hypothesis of the theorem.
\end{proof}

\end{document}